\documentclass[10pt]{amsart}
\usepackage{amsmath,amscd,amssymb,amsthm,amsbsy,amscd,stmaryrd}
\usepackage {hyperref}

\newtheorem{theorem}{Theorem}[section]
\newtheorem{lemma}[theorem]{Lemma}
\newtheorem{proposition}[theorem]{Proposition}
\newtheorem{corollary}[theorem]{Corollary}
\newtheorem{remark}[theorem]{Remark}

\newtheorem{conjecture}[theorem]{Conjecture}

\newtheorem{question}[theorem]{Question}
\theoremstyle{definition}

\numberwithin{equation}{section}

\begin{document}
\title{A geometric heat flow for vector fields}
\author{Yi Li}
\address{Department of Mathematics, Shanghai Jiao Tong University,
800 Dong Chuan Road, Ming Hang District, Shanghai, 200240, China}
\email{yilicms@gmail.com}

\author{Kefeng Liu}
\address{Department of Mathematics, UCLA, Los Angeles, CA 90095-1555; Center of Mathematical Sciences, Zhejiang University, Hangzhou, China 310027}
\email{liu@math.ucla.edu}

\subjclass[2010]{Primary 53C44, 35K55}

\keywords{Geometric heat flow, Killing vector fields,
Yano's theorem, Navier-Stokes equations, Kazdan-Warner-Bourguignon-Ezin
identity}

\begin{abstract} In this paper we introduce and
study a geometric heat flow to find Killing vector fields on closed Riemannian manifolds with positive sectional curvature. We study its various properties, prove the global existence of the solution of this flow, discuss its convergence and possible applications, and its relation to the Navier-Stokes equations on manifolds and Kazdan-Warner-Bourguignon-Ezin identity
 for conformal Killing vector fields. We also provide two new criterions on 
 the existence of Killing vector 
 fields. The similar flow to finding
holomorphic vector fields on K\"ahler manifolds will be studied in \cite{LL2}. 
\end{abstract}

\maketitle

\section{A geometric heat flow for vector fields}\label{section1}

Recently, we have witnessed the power of geometric flows in studying lots of
problems in geometry and topology. In this paper we introduce a geometric heat flow for vector fields on a Riemannian manifold and study its varies properties.

Throughout this paper, we adopt the Einstein
summation and notions as those in \cite{CLN}. All manifolds and vector
fields are smooth; a manifold is said to be {\it closed} if it is compact
and without boundary. We shall often raise and lower indices for tensor
fields.

\subsection{Deformation tensor field of a vector field}\label{subsection1.1}

Let $(M,g)$ be a closed and orientable Riemannian manifold. To a
vector field $X$ we associate its deformation $(0,2)$-tensor field
${\bf Def}(X)$, which is an obstruction of $X$ to be Killing and is locally defined by
\begin{equation}
\left({\bf Def}(X)\right)_{ij}:=
\frac{\nabla_{i}X_{j}+\nabla_{j}X_{i}}{2},\label{1.1}
\end{equation}
where $\nabla$ denotes the Levi-Civita connection of $g$. Equivalently, it is exactly (up to a constant factor) the Lie derivative of $g$ along the vector field $X$, i.e., $\mathcal{L}_{X}g$. We say that $X$ is a {\it Killing vector field} if ${\bf Def}(X)=0$. Consider the
$L^{2}$-norm of ${\bf Def}(X)$:
\begin{equation}
\mathfrak{L}(X):=\int_{M}|{\bf Def}(X)|^{2}dV,\label{1.2}
\end{equation}
where $dV$ stands for the volume form of $g$ and $|\cdot|$ means the
norm of ${\bf Def}(X)$ with respect to $g$. Clearly that the critical point $X$ of $\mathfrak{L}$ satisfies
\begin{equation}
\Delta X^{i}+\nabla^{i}{\rm div}(X)+R^{i}{}_{j}X^{j}=0.\label{1.3}
\end{equation}
Here and henceforth, $\Delta:= g^{ij}\nabla_{i}\nabla_{j}$ is
the Laplace-Beltrami operator of $g$ and $R_{ij}$ denotes the Ricci curvature
of $g$. In fact
\begin{eqnarray*}
\frac{d}{dt}\mathfrak{L}(X_{t})&=&\frac{1}{2}\int_{M}
\langle{\bf Def}(X_{t}),\partial_{t}{\bf Def}(X_{t})\rangle\!\ dV\nonumber\\
&=&\frac{1}{2}\int_{M}\left(\nabla^{i}(X_{t})^{j}+\nabla^{j}(X_{t})^{i}\right)
\left(\nabla_{i}\partial_{t}(X_{t})_{j}+\nabla_{j}\partial_{t}(X_{t})_{i}
\right)dV\nonumber\\
&=&-\int_{M}\left[\Delta(X_{t})^{i}\cdot\partial_{t}(X_{t})_{i}
+\nabla^{j}\nabla^{i}(X_{t})^{j}\cdot\partial_{t}(X_{t})_{i}\right]dV\nonumber\\
&=&-\int_{M}\left[\Delta(X_{t})^{i}
+\nabla^{i}{\rm div}(X_{t})+R^{i}{}_{j}(X_{t})^{j}\right]
\partial_{t}(X_{t})^{i}dV.
\end{eqnarray*}

\subsection{A geometric heat flow for vector fields}\label{subsection1.2}

Motivated by (\ref{1.3}), we introduce a geometric heat flow for vector fields:
\begin{equation}
\partial_{t}(X_{t})^{i}=\Delta (X_{t})^{i}+\nabla^{i}{\rm div}(X_{t})+R^{i}{}_{j}(X_{t})^{j}, \ \ \ X_{0}=X,\label{1.4}
\end{equation}
where $X$ is a fixed vector field on $M$ and $\partial_{t}:=\frac{\partial}{\partial t}$ is the time derivative. If we define
${\rm Ric}^{\sharp}$, the $(1,1)$-tensor field associated to ${\rm Ric}$, by
\begin{equation*}
g\left({\rm Ric}^{\sharp}(X),Y\right):=t{\rm Ric}(X,Y),
\end{equation*}
where $X, Y$ are two vector fields, then ${\rm Ric}^{\sharp}$ is an operator
on the space of vector fields, denoted by $C^{\infty}(TM)$, and the
flow (\ref{1.4}) can be rewritten as
\begin{equation}
\partial_{t}X_{t}=\Delta X_{t}+\nabla{\rm div}(X_{t})
+{\rm Ric}^{\sharp}(X_{t}).\label{1.5}
\end{equation}

In 1952, Yano (e.g., \cite{Y1,Y4, YB}) showed that a vector field $X=X^{i}\frac{\partial}{\partial x^{i}}$ is a Killing vector field if and only if it satisfies
\begin{equation}
\Delta X^{i}+R^{i}{}_{j}X^{j}=0, \ \ \ {\rm div}(X)=0.\label{1.6}
\end{equation}
His result depends on an integral formula, now called {\it Yano's integral formula},
\begin{equation}
0=\int_{M}\left[{\rm Ric}(X,X)-|\nabla X|^{2}+
2|{\bf Def}(X)|^{2}-|{\rm div}(X)|^{2}\right]dV,\label{1.7}
\end{equation}
which holds for any vector field $X$. This integral formula lets us define
so-called the {\it Bochner-Yano integral} for every vector field $X$:
\begin{equation}
\mathcal{E}(X):=\int_{M}\left[|\nabla X|^{2}+|{\rm div}(X)|^{2}
-{\rm Ric}(X,X)\right]dV.\label{1.8}
\end{equation}
Consequently, Yano's integral formula implies that
$\mathcal{E}(X)$ is always nonnegative and $\mathcal{E}(X)=2\mathfrak{L}(X)$
for every vector field $X$. On the other hand, Watanabe \cite{W} proved that $X$ is a Killing vector field if and only if $\mathcal{E}(X)=0$, and hence if and only if $\mathfrak{L}(X)=0$.

Yano's equations (\ref{1.6}) induces a system of equations, called the
{\it Bochner-Yano flow}:
\begin{equation}
\partial_{t}(X_{t})^{i}=\Delta(X_{t})^{i}+R^{i}{}_{j}(X_{t})^{j}, \ \ \
{\rm div}(X_{t})=0.\label{1.9}
\end{equation}
Notice that Yano's equation (\ref{1.6}) (resp., Bochner-Yano flow (\ref{1.9})) is a special case of our equation (\ref{1.3}) (resp., our flow (\ref{1.4})).

\begin{proposition}\label{p1.1} If $X_{t}$ is the solution of the flow (\ref{1.4}), then
\begin{eqnarray}
\mathcal{E}(X_{t})&\geq&0,\label{1.10}\\
\frac{d}{dt}\mathcal{E}(X_{t})&=&-2\int_{M}|\partial_{t}X_{t}|^{2}dV \ \ \leq \ \ 0,\label{1.11}\\
\mathcal{E}(X_{t})&=&-\frac{d}{dt}\left(\frac{1}{2}\int_{M}|X_{t}|^{2}dV\right).\label{1.12}
\end{eqnarray}
Consequently, $\mathcal{E}(X_{t})$ is monotone nonincreasing and $\int_{M}|X_{t}|^{2}dV$ is also monotone nonincreasing.
\end{proposition}

\begin{proof} The first one directly follows from (\ref{1.7}). Since the flow (\ref{1.4}) is the gradient flow of the functional $\mathcal{E}$, we prove the second one. To prove (\ref{1.12}), we use the formula $\frac{1}{2}\Delta|X|^{2}=\left\langle X,\Delta X\right\rangle+|\nabla X|^{2}$ to deduce that
\begin{eqnarray*}
\mathcal{E}(X_{t})&=&\int_{M}\left[\frac{1}{2}\Delta|X_{t}|^{2}-(X_{t})_{i}\Delta(X_{t})^{i}+|{\rm div}(X_{t})|^{2}-{\rm Ric}(X_{t},X_{t})\right]dV\\
&=&-\int_{M}\left[(X_{t})_{i}\Delta(X_{t})^{i}
-{\rm div}(X_{t})\cdot{\rm div}(X_{t})+{\rm Ric}(X_{t},X_{t})\right]dV\\
&=&-\int_{M}\left[(X_{t})_{i}\Delta(X_{t})^{i}+(X_{t})_{i}\nabla^{i}{\rm div}(X_{t})
+(X_{t})_{i}\cdot R^{i}{}_{j}(X_{t})^{j}\right]dV\\
&=&-\int_{M}(X_{t})_{i}\left[\Delta(X_{t})^{i}+\nabla^{i}{\rm div}(X_{t})
+R^{i}{}_{j}(X_{t})^{j}\right]dV\\
&=&-\frac{1}{2}\int_{M}\partial_{t}
|X_{t}|^{2}dV.
\end{eqnarray*}
Hence the conclusion is obvious.
\end{proof}

\begin{corollary} \label{c1.2} If $X_{t}$ is the solution of the flow
(\ref{1.4}) for $t\in[0,T]$, then we have
\begin{equation}
\int^{T}_{0}\int_{M}|\partial_{t}X_{t}|^{2}dVdt\leq 2\mathcal{E}
(X).\label{1.13}
\end{equation}
\end{corollary}

\begin{proof} For any $T$, we have
\begin{equation*}
-\frac{1}{2}\int^{T}_{0}\int_{M}|\partial_{t}X_{t}|^{2}dVdt=\mathcal{E}(X_{T})-\mathcal{E}(X)
\geq-\mathcal{E}(X)
\end{equation*}
since $\mathcal{E}$ is nonnegative. This proves (\ref{1.13}).
\end{proof}

\subsection{Evolution equations}\label{subsection1.3}

To study the long time existence and the convergence of the geometric heat
flow (\ref{1.4}) we prove its several evolution equations.

\begin{lemma} \label{l1.3}If $X_{t}$ is the solution of (\ref{1.4}), then
\begin{equation}
\partial_{t}|X_{t}|^{2}=\Delta|X_{t}|^{2}-2|\nabla X_{t}|^{2}+
2\left\langle X_{t},\nabla{\rm div}(X_{t})\right\rangle+2{\rm Ric}(X_{t},
X_{t}).\label{1.14}
\end{equation}
\end{lemma}

\begin{proof} Calculate
\begin{eqnarray*}
\partial_{t}|X_{t}|^{2}&=&2(X_{t})_{i}\partial_{t}(X_{t})^{i}\\
&=&2(X_{t})_{i}\left(\Delta(X_{t})^{i}+\nabla^{i}{\rm div}(X_{t})
+R^{i}{}_{j}(X_{t})^{j}\right)\\
&=&\Delta|X_{t}|^{2}-2|\nabla X_{t}|^{2}+
2\left\langle X_{t},\nabla{\rm div}(X_{t})\right\rangle
+2{\rm Ric}(X_{t},X_{t})
\end{eqnarray*}
which proves (\ref{1.4}).
\end{proof}

\begin{lemma} \label{l1.4} If $X_{t}$ is the solution of (\ref{1.4}), then
\begin{eqnarray}
\partial_{t}|\nabla X_{t}|^{2}&=&\Delta|\nabla X_{t}|^{2}
-2\left|\nabla^{2}X_{t}\right|^{2}-4R_{ijk\ell}\nabla^{i}(X_{t})^{k}\cdot
\nabla^{j}(X_{t})^{\ell}\nonumber\\
&&- \ 2R_{ij}\nabla^{i}(X_{t})^{k}\cdot\nabla^{j}(X_{t})_{k}+2R_{ij}\nabla^{k}(X_{t})^{i}\cdot \nabla_{k}(X_{t})^{j}\nonumber\\
&&+ \ 2\left\langle{\bf Def}(X_{t}),\nabla\nabla{\rm div}(X_{t})\right\rangle\label{1.15}\\
&&+ \ 2\left(\nabla_{i}R_{jk}-\nabla^{\ell}R_{\ell ikj}\right)(X_{t})^{k}\nabla^{i}(X_{t})^{j}.\nonumber
\end{eqnarray}
\end{lemma}

\begin{proof} From the definition of the flow we have
\begin{eqnarray*}
\partial_{t}|\nabla X_{t}|^{2}&=&2\nabla^{i}(X_{t})_{j}\cdot\nabla_{i}\partial_{t}
(X_{t})^{j}\\
&=&2\nabla^{i}(X_{t})_{j}\cdot\nabla_{i}\left(\Delta(X_{t})^{j}
+\nabla^{j}{\rm div}(X_{t})+R^{j}{}_{k}(X_{t})^{k}\right).
\end{eqnarray*}
We use the Ricci identity to deduce that
\begin{eqnarray*}
\nabla_{i}\Delta(X_{t})^{j}&=&\nabla_{i}\left(g^{pq}\nabla_{p}
\nabla_{q}(X_{t})^{j}\right)\\
&=&g^{pq}\nabla_{i}\nabla_{p}\nabla_{q}(X_{t})^{j}\\
&=&g^{pq}\left[\nabla_{p}\nabla_{i}\nabla_{q}(X_{t})^{j}
-R_{ipq}{}^{r}\nabla_{r}(X_{t})^{j}+R_{ipr}{}^{j}\nabla_{q}(X_{t})^{r}\right]\\
&=&\nabla^{q}\left[\nabla_{q}\nabla_{i}(X_{t})^{j}+R_{iqr}{}^{j}
(X_{t})^{r}\right]-R_{ir}\nabla^{r}(X_{t})^{j}+R_{ipr}{}^{j}\nabla^{p}(X_{t})^{r}\\
&=&\Delta\nabla_{i}(X_{t})^{j}+\nabla^{q}\left(R_{iqr}{}^{j}
(X_{t})^{r}\right)-R_{ir}\nabla^{r}(X_{t})^{j}+R_{ipr}{}^{j}\nabla^{p}(X_{t})^{r}\\
&=&\Delta\nabla_{i}(X_{t})^{j}+\nabla^{q}R_{iqr}{}^{j}\cdot(X_{t})^{r}
+2R_{iqr}{}^{j}\nabla^{q}(X_{t})^{r}-R_{ir}\nabla^{r}(X_{t})^{j}.
\end{eqnarray*}
Plugging it into the equation for $\partial_{t}|\nabla X_{t}|^{2}$ we arrive at
\begin{eqnarray*}
\partial_{t}|\nabla X_{t}|^{2}&=&2\nabla^{i}(X_{t})_{j}\Big[\Delta\nabla_{i}(X_{t})^{j}
+\nabla^{q}R_{iq r}{}^{j}(X_{t})^{r}+2R_{iqr}{}^{j}\nabla^{q}(X_{t})^{r}\\
&&- \ R_{ir}\nabla^{r}(X_{t})^{j}+\nabla_{i}\nabla^{j}{\rm div}(X_{t})+(X_{t})^{k}\nabla_{i}R^{j}{}_{k}+R^{j}{}_{k}\nabla_{i}(X_{t})^{k}\Big]\\
&=&\Delta|\nabla X_{t}|^{2}-2\left|\nabla^{2}X_{t}\right|^{2}
+2\nabla^{q}R_{iqr}{}^{j}\nabla^{i}(X_{t})_{j}\cdot(X_{t})^{r}\\
&&+ \ 4R_{iqrj}\nabla^{q}(X_{t})^{r}\nabla^{i}(X_{t})^{j}-2R_{ir}
\nabla^{r}(X_{t})^{j}\nabla^{i}(X_{t})_{j}\\
&&+ \ 2\nabla^{i}(X_{t})_{j}\cdot
\nabla_{i}\nabla^{j}{\rm div}(X_{t})
+2\nabla_{i}R^{j}{}_{k}\cdot(X_{t})^{k}\nabla^{i}(X_{t})_{j}\\
&& \ 2R^{j}{}_{k}\nabla^{i}(X_{t})^{k}\nabla^{i}(X_{t})_{j}\\
&=&\Delta|\nabla X_{t}|^{2}-2\left|\nabla^{2}X_{t}\right|^{2}-4
R_{qirj}\nabla^{q}(X_{t})^{r}\nabla^{i}(X_{t})^{j}\\
&&- \ 2R_{ir}\nabla^{r}(X_{t})^{j}\nabla^{i}(X_{t})_{j}+2R_{jk}\nabla_{i}(X_{t})^{k}\nabla^{i}(X_{t})^{j}\\
&&+ \ 2\nabla^{i}(X_{t})_{j}\cdot\nabla_{i}\nabla^{j}{\rm div}(X_{t})+2\nabla_{i}R_{jk}\cdot(X_{t})^{k}\nabla^{i}(X_{t})^{j}\\
&& - \ 2\nabla^{q}R_{qirj}(X_{t})^{r}
\nabla^{i}(X_{t})^{j}.
\end{eqnarray*}
Changing the indices yields the desired result.
\end{proof}

By the Bianchi identity, the above lemma can be written as

\begin{corollary} \label{c1.5} If $X_{t}$ is the solution of the flow
(\ref{1.4}), then
\begin{eqnarray*}
\partial_{t}|\nabla X_{t}|^{2}&=&\Delta|\nabla X_{t}|^{2}-2\left|\nabla^{2}X_{t}\right|^{2}
-4R_{ijk\ell}\nabla^{i}(X_{t})^{k}\nabla^{j}(X_{t})^{\ell}\\
&&- \ 2R_{ij}\nabla^{i}(X_{t})^{k}\nabla^{j}(X_{t})_{k}+2R_{ij}\nabla^{k}(X_{t})^{i}
\nabla_{k}(X_{t})^{j}\\
&&+ \ 2\left(\nabla_{i}R_{jk}-\nabla_{j}R_{ki}+\nabla_{k}R_{ij}\right)
(X_{t})^{k}\nabla^{i}(X_{t})^{j}\\
&&+ \ 2\left\langle{\bf Def}(X_{t}),\nabla
\nabla{\rm div}(X_{t})\right\rangle.
\end{eqnarray*}
\end{corollary}

\begin{lemma} \label{l1.6} (1) If $X_{t}$ is the solution of the flow
(\ref{1.4}), then
\begin{equation}
\partial_{t}{\rm div}(X_{t})=2\Delta{\rm div}(X_{t})
+\langle X_{t},\nabla R\rangle+2R_{ij}\nabla^{i}(X_{t})^{j}\label{1.16}
\end{equation}
(2) If $X_{t}$ is the solution of the flow (\ref{1.4}), then
\begin{eqnarray*}
\partial_{t}|{\rm div}(X_{t})|^{2}&=&2\Delta|{\rm div}(X_{t})|^{2}
-4\left|\nabla{\rm div}(X_{t})\right|^{2}\\
&&+ \ 2{\rm div}(X_{t})\langle X_{t},\nabla R\rangle+4{\rm div}(X_{t})\cdot R_{ij}\nabla^{i}(X_{t})^{j}
\end{eqnarray*}
and
\begin{eqnarray}
\frac{d}{dt}\int_{M}|{\rm div}(X_{t})|^{2}dV&=&
-4\int_{M}|\nabla{\rm div}(X_{t})|^{2}dV\nonumber\\
&&- \ 4\int_{M}{\rm Ric}\left(X_{t},\nabla{\rm div}(X_{t})\right)dV.\label{1.17}
\end{eqnarray}
In particular, if ${\rm Ric}=0$ and ${\rm div}(X)\equiv0$, then ${\rm div}(X_{t})\equiv0$.
\end{lemma}

\begin{proof} According to (\ref{1.4}), one has
\begin{eqnarray*}
\partial_{t}{\rm div}(X_{t})&=&\nabla_{i}\left(\partial_{t}(X_{t})^{i}\right) \ \ = \ \ \nabla_{i}
\left(\Delta(X_{t})^{i}+\nabla^{i}{\rm div}(X_{t})+R^{i}{}_{j}(X_{t})^{j}\right)\nonumber\\
&=&\nabla_{i}\left(\Delta(X_{t})^{i}\right)+\Delta{\rm div}(X_{t})+\nabla^{i}\left(R_{ij}(X_{t})^{j}\right).
\end{eqnarray*}
Next we compute the first term $\nabla_{i}\left(\Delta(X_{t})^{i}\right)$ as follows:
\begin{eqnarray*}
\nabla_{i}\left(\Delta(X_{t})^{i}\right)&=&g^{pq}\nabla_{i}\nabla_{p}\nabla_{q}(X_{t})^{i}\\
&=&g^{pq}\left(\nabla_{p}\nabla_{i}\nabla_{q}(X_{t})^{i}-R_{ipq}{}^{r}\nabla_{r}(X_{t})^{i}
+R_{ipr}{}^{i}\nabla_{q}(X_{t})^{r}\right)\\
&=&\nabla^{q}\left(\nabla_{q}\nabla_{i}(X_{t})^{i}+R_{iqr}{}^{i}
(X_{t})^{r}\right)
-R_{ir}\nabla^{r}(X_{t})^{i}+R_{pr}\nabla^{p}(X_{t})^{r}\\
&=&\Delta\nabla_{i}(X_{t})^{i}+\nabla^{q}\left(R_{qr}(X_{t})^{r}\right).
\end{eqnarray*}
Combining those two expression gives
\begin{eqnarray*}
\partial_{t}{\rm div}(X_{t})&=&2\Delta{\rm div}(X_{t})+2\nabla^{i}\left(R_{ij}(X_{t})^{j}\right)\\
&=&2\Delta{\rm div}(X_{t})+2\nabla^{i}R_{ij}\cdot(X_{t})^{j}+2R_{ij}\nabla^{i}(X_{t})^{j}\\
&=&2\Delta{\rm div}(X_{t})+\nabla_{j}R\cdot(X_{t})^{j}+2R_{ij}\nabla^{i}(X_{t})^{j}
\end{eqnarray*}
proving (\ref{1.16}). For (\ref{1.17}), the evolution equation for $|{\rm div}(X_{t})|^{2}$ is
\begin{eqnarray*}
\partial_{t}|{\rm div}(X_{t})|^{2}&=&2{\rm div}(X_{t})\cdot\partial_{t}{\rm div}(X_{t})\\
&=&2{\rm div}(X_{t})\left(2\Delta{\rm div}(X_{t})+(X_{t})^{i}\nabla_{i}R
+2R_{ij}\nabla^{i}(X_{t})^{j}\right)\\
&=&2\Delta|{\rm div}(X_{t})|^{2}-4|\nabla{\rm div}(X_{t})|^{2}\\
&&+ \ 2{\rm div}(X_{t})\cdot(X_{t})^{i}\nabla_{i}R
+4\left({\rm div}(X_{t})R_{ij}\right)\nabla^{i}(X_{t})^{j}.
\end{eqnarray*}
Integrating both sides over $M$ yields
\begin{eqnarray*}
&&\frac{d}{dt}\int_{M}|{\rm div}(X_{t})|^{2}dV \ \ = \ \ - \ 4\int_{M}|\nabla{\rm div}
(X_{t})|^{2}dV\\
&&+ \ 2\int_{M}{\rm div}(X_{t})\left((X_{t})^{i}\nabla_{i}R\right)dV
-4\int_{M}\nabla^{i}\left({\rm div}(X_{t})R_{ij}\right)(X_{t})^{j}dV.
\end{eqnarray*}
Since
\begin{eqnarray*}
4\nabla^{i}\left({\rm div}(X_{t})R_{ij}\right)(X_{t})^{j}
&=&4\left[\nabla^{i}{\rm div}(X_{t})\cdot R_{ij}+{\rm div}(X_{t})\cdot\nabla^{i}R_{ij}\right](X_{t})^{j}\\
&=&4R_{ij}(X_{t})^{j}\nabla^{i}{\rm div}(X_{t})+2\nabla_{j}R\cdot (X_{t})^{j}{\rm div}(X_{t})
\end{eqnarray*}
it follows that (\ref{1.17}) is true. When ${\rm Ric}=0$, we obtain
\begin{equation*}
\frac{d}{dt}\int_{M}|{\rm div}(X_{t})|^{2}dV\leq0
\end{equation*}
which means $\int_{M}|{\rm div}(X_{t})|^{2}dV\leq\int_{M}|{\rm div}(X)|^{2}dV=0$ and therefore $|{\rm div}(X_{t})|^{2}=0$. Thus ${\rm div}(X_{t})\equiv0$.
\end{proof}

\subsection{Long-time existence}\label{subsection1.4}

Now we can state our main results to the flow (\ref{1.4}).

\begin{theorem} \label{t1.7} {\bf (Long-time existence)}Suppose that $(M,g)$
is a closed and orientable Riemannian manifold. Given an initial vector
field, the flow (\ref{1.4}) exists for all time.
\end{theorem}

The main method on proving above theorem is the standard approach in PDEs and an application of Sobolev embedding theorem. After establishing the long-time existence, we can study the convergence problem of
the flow (\ref{1.4}).

\begin{proof} We now turn to prove the short-time existence of the flow (\ref{1.4}). Note that (\ref{1.4}) can be written as
\begin{eqnarray}
\partial_{t}(X_{t})^{i}&=&\Delta(X_{t})^{i}
+\nabla^{i}{\rm div}(X_{t})+R^{i}{}_{j}(X_{t})^{j}\nonumber\\
&=&\nabla^{k}\nabla_{k}(X_{t})^{i}+\nabla^{i}\nabla_{j}(X_{t})^{j}
+R^{i}{}_{j}(X_{t})^{j}\label{1.18}\\
&=&\sum^{m}_{j=1}\left(\delta^{i}_{j}\sum^{m}_{k=1}\nabla^{k}\nabla_{k}+\nabla^{i}\nabla_{j}\right)
(X_{t})^{j}+R^{i}{}_{j}(X_{t})^{j}.\nonumber
\end{eqnarray}
For any $\xi=(\xi_{1},\cdots,\xi_{m})\in\mathbb{R}^{m}$, we have
\begin{equation}
\sum^{m}_{i,j=1}\left(\delta^{i}_{j}\sum^{m}_{k=1}
\xi_{k}\xi_{k}+\xi_{i}\xi_{j}\right)=\sum^{m}_{i,k=1}\xi_{k}\xi_{k}
+\sum^{m}_{i,j=1}\xi_{i}\xi_{j}=m|\xi|^{2}+\sum^{m}_{i,j=1}\xi_{i}\xi_{j},
\label{1.19}
\end{equation}
where $|\xi|\doteqdot\left(\sum^{m}_{k=1}\xi^{2}_{k}\right)^{1/2}$ denotes the
length of $\xi$ in $\mathbb{R}^{m}$. On the other hand, plugging
\begin{equation*}
\sum^{m}_{i,j=1}(\xi_{i}+\xi_{j})^{2}=\sum^{m}_{i,j=1}\left(\xi^{2}_{i}
+\xi^{2}_{j}+2\xi_{i}\xi_{j}\right)=2m|\xi|^{2}+2\sum^{m}_{i,j=1}\xi_{i}\xi_{j}
\end{equation*}
into (\ref{1.19}) yields
\begin{eqnarray*}
\sum^{m}_{i,j=1}\left(\delta^{i}_{j}\sum^{m}_{k=1}\xi_{k}\xi_{k}
+\xi_{i}\xi_{j}\right)&=&\frac{1}{2}\sum^{m}_{i,j=1}(\xi_{i}+\xi_{j})^{2}\\
&=&2\sum^{m}_{i=1}\xi^{2}_{i}+\frac{1}{2}\sum_{i\neq j}(\xi_{i}+\xi_{j})^{2}\\
&\geq&2|\xi|^{2}.
\end{eqnarray*}
Then, by the standard theory for partial differential equations of parabolic
type, we have that the flow (\ref{1.4}) exists for a short time.

Since the flow equation is linear, a standard theory in PDEs implies the
long-time existence.
\end{proof}

\subsection{Convergence}\label{subsection1.5}

In what follows, we always assume that $(M,g)$ is a closed and oriented
Riemannian manifold of dimension $m$. Since $M$ is compact, we can find a constant $B$ such that
\begin{equation}
R_{ij}\leq B g_{ij}.\label{1.20}
\end{equation}
Then the energy functional $\mathcal{E}(X_{t})$ satisfies
\begin{equation}
\int_{M}\left[|\nabla X_{t}|^{2}+\left({\rm div}(X_{t})\right)^{2}-B|X_{t}|^{2}\right]dV
\leq\mathcal{E}(X_{t}).\label{1.21}
\end{equation}
Using Proposition \ref{p1.1}, we have
\begin{equation*}
\int_{M}|\nabla X_{t}|^{2}dV\leq\mathcal{E}(X_{t})+B
\int_{M}|X_{t}|^{2}dV=\mathcal{E}(X_{t})+B\cdot u(t)\leq\mathcal{E}(X)+B\cdot u(0),
\end{equation*}
where
\begin{equation*}
u(t):=\int_{M}|X_{t}|^{2}dV.
\end{equation*}
Hence $\nabla X_{t}\in L^{2}(M,TM)$. On the other hand $u(t)\leq u(0)$, we
conclude that
\begin{equation}
||X_{t}||_{H^{1}(M,TM)}\leq C_{1}(M,g,X).\label{1.22}
\end{equation}
By the regularity of parabolic equations and the flow
(\ref{1.4}), we obtain
\begin{equation*}
||X_{t}||_{H^{\ell}(M,TM)}\leq C_{\ell}=C_{\ell}(M,g,X)
\end{equation*}
for each $\ell$. Therefore we can find $X_{\infty}\in H^{\ell}(M,TM)$ and a subsequence $(X_{t_{i}})_{i\in\mathbb{N}}$ such that $X_{t_{i}}\to X_{\infty}$ a.e. as $i\to\infty$. By Sobolev imbedding theorem, $X_{\infty}\in C^{\infty}(M,TM)$ and $X_{t}\to
X_{\infty}$ as $t\to\infty$.

Corollary \ref{c1.2} implies there exists a subsequence, say,
without loss of generality, $(X_{t_{i}})_{i\in\mathbb{N}}$, such that
\begin{equation}
\left|\left|\partial_{t}X_{t}\big|_{t=t_{i}}\right|\right|_{L^{2}
(M,g)}\to0.\label{1.23}
\end{equation}
According to (\ref{1.11}) and (\ref{1.23}), $||\partial_{t}X_{t}||_{L^{2}
(M,g)}$ decreases and converges to $0$ as $t\to\infty$. Therefore the smooth vector
field $X_{\infty}$ satisfies
\begin{equation}
\Delta_{{\rm LB}}(X_{\infty})^{i}+\nabla^i{\rm{div}}(X_{\infty})+R^{i}{}_{j}
(X_{\infty})^{j}=0.\label{1.24}
\end{equation}

In summary, we proved

\begin{theorem} \label{t1.8} {\bf (Convergence)}Suppose that $(M,g)$ is a closed and orientable Riemannian manifold. If $X$ is a vector field, there exists a unique smooth solution $X_{t}$ to the flow (\ref{1.4}) for all time $t$. As $t$ goes to infinity, the vector field
$X_{t}$ converges uniformly to a Killing vector field $X_{\infty}$.
\end{theorem}

\begin{remark} \label{r1.9} As Professor Cliff Taubes remarked that Theorem
\ref{t1.7}
and \ref{t1.8} also follow from an eigenfunction expansion for the
relevant linear operator that defines the flow (\ref{1.4}), which gives a short proof
of those two theorems.
\end{remark}

Theorem \ref{t1.8} does {\it not} give us a nontrivial Killing vector field. For
example, if $X$ is identically zero, then by the uniqueness theorem the limit
vector field is also identically zero. When the Ricci curvature is negative,
Bochner's theorem implies that there is no nontrivial Killing vector field.

To obtain a nonzero Killing vector field, we have the following criterion.

\begin{proposition} \label{p1.10}Suppose that $(M,g)$ is a closed and orientable Riemannian manifold and $X$ is a vector field on $M$.
If $X_{t}$ is the solution of the flow (\ref{1.4}) with
the initial value $X$, then
\begin{equation}
\int^{\infty}_{0}\mathcal{E}(X_{t})\!\ dV<\infty.\label{1.25}
\end{equation}
Let
\begin{equation}
\mathbf{Err}(X):=\frac{1}{2}\int_{M}|X|^{2}dV
-\int^{\infty}_{0}\mathcal{E}(X_{t})\!\ dt.\label{1.26}
\end{equation}
Therefore $\mathbf{Err}(X)\geq0$ and $X_{\infty}$ is nonzero if and only if
$\mathbf{Err}(X)>0$.
\end{proposition}

The higher derivatives of $\mathcal{E}(X_{t})$ have explicit formulas in terms
of the energy functionals of lower derivatives of $X_{t}$.

\begin{proposition} \label{p1.11}If $X_{t}$ is the solution of the flow (\ref{1.4}), then
\begin{equation}
\mathcal{E}''(X_{t})=4\mathcal{E}(\partial_{t}X_{t})\geq0.\label{1.27}
\end{equation}
\end{proposition}

\begin{proof} Using (\ref{1.11}), we have
\begin{eqnarray*}
\mathcal{E}''(X_{t})&=&-4\int_{M}\partial_{t}(X_{t})_{i}\cdot\partial_{t}
\left(\partial_{t}(X_{t})^{i}\right)dV\\
&=&-4\int_{M}\partial_{t}(X_{t})_{i}\cdot\partial_{t}
\left(\Delta_{{\rm LB}}(X_{t})^{i}+\nabla^{i}{\rm div}(X_{t})
+R^{i}{}_{j}(X_{t})^{j}\right)dV\\
&=&-4\int_{M}\partial_{t}(X_{t})_{i}\left(\Delta_{{\rm LB}}\partial_{t}(X_{t})^{i}
+\nabla^{i}{\rm div}(\partial_{t}X_{t})+R^{i}{}_{j}\partial_{t}(X_{t})^{j}\right)dV\\
&=&-4\int_{M}\left(\frac{1}{2}\Delta|\partial_{t}X_{t}|^{2}
-|\nabla\partial_{t}X_{t}|^{2}\right)dV\\
&&- \ 4\int_{M}\partial_{t}(X_{t})_{i}\left(\nabla^{i}{\rm div}(\partial_{t}X_{t})+R^{i}{}_{j}
\partial_{t}(X_{t})^{j}\right)dV\\
&=&4\int_{M}\left[|\nabla\partial_{t}X_{t}|^{2}+\nabla^{i}\partial_{t}(X_{t})_{i}
\cdot{\rm div}(\partial_{t}X_{t})-{\rm Ric}(\partial_{t}X_{t},\partial_{t}X_{t})\right]dV\\
&=&4\int_{M}|\left[\nabla\partial_{t}X_{t}|^{2}+\partial_{t}{\rm div}(X_{t})
\cdot{\rm div}(\partial_{t}X_{t})
-{\rm Ric}(\partial_{t}X_{t},\partial_{t}X_{t})\right]dV\\
&=&4\int_{M}\left[|\nabla\partial_{t}X_{t}|^{2}
+({\rm div}(\partial_{t}X_{t}))^{2}-{\rm Ric}(\partial_{t}X_{t},\partial_{t}X_{t})\right]dV\\
&=&4\mathcal{E}(\partial_{t}X_{t})
\end{eqnarray*}
which is nonnegative according to (\ref{1.7}).
\end{proof}

\subsection{A connection to the Navier-Stokes equations}\label{subsection1.6}

A surprising observation is that our flow (\ref{1.4}) is very close to
the Navier-Stokes equations \cite{CJW, T}(without the pressure) on manifolds
\begin{equation}
\partial_{t}X_{t}+\nabla_{X_{t}}X_{t}={\rm div}(S_{t}), \ \ \ {\rm div}(X_{t})=0,
\label{1.28}
\end{equation}
where $S_{t}:=2{\bf Def}(X_{t})$ is the stress tensor of $X_{t}$. By an
easy computation we can write (\ref{1.28}) as
\begin{equation}
\partial_{t}(X_{t})^{i}+\left(\nabla_{X_{t}}X_{t}\right)^{i}=\Delta
(X_{t})^{i}+\nabla^{i}{\rm div}(X_{t})+R^{i}{}_{j}X^{j}, \ \ \
{\rm div}(X_{t})=0.\label{1.29}
\end{equation}
Compared (\ref{1.4}) with (\ref{1.29}), we give a geometric interpolation
of the right (or the linear) part of the Navier-Stokes equations on manifolds.

When the Ricci tensor field is identically zero, our flow (\ref{1.4}) keeps the property that
${\rm div}(X_{t})=0$ (see (\ref{1.17})).

As a consequence of the non-negativity of $\mathcal{E}$ we can prove that

\begin{theorem} \label{t1.12}Suppose that $(M,g)$ is a closed and orientable Riemannian manifold. If $X_{t}$ is a solution of the Navier-Stokes equations (\ref{1.35}), then
\begin{equation}
\frac{d}{dt}\left(\int_{M}|X_{t}|^{2}dV\right)=-2\mathcal{E}(X_{t})
\leq0.\label{1.30}
\end{equation}
In particular
\begin{equation}
\int_{M}|X_{t}|^{2}dV\leq\int_{M}|X_{0}|^{2}dV.\label{1.31}
\end{equation}
\end{theorem}

\begin{proof} By multiplying by $(X_{t})_{i}$ the equation (\ref{1.29}) equals
\begin{equation*}
\frac{1}{2}\partial_{t}|X_{t}|^{2}+\left\langle\nabla_{X_{t}}X_{t},X_{t}\right\rangle
=\left\langle \Delta X_{t}+\nabla{\rm div}(X_{t})+{\rm Rc}^{\sharp}(X_{t}),X_{t}\right\rangle.
\end{equation*}
Integrating on both sides yields
\begin{equation*}
\frac{1}{2}\frac{d}{dt}\int_{M}|X_{t}|^{2}dV+\int_{M}\left\langle\nabla_{X_{t}}X_{t},X_{t}\right\rangle dV
=-\mathcal{E}(X_{t}).
\end{equation*}
From Lemma \ref{l1.13} below, we verify (\ref{1.30}) since ${\rm div}(X_{t})=0$.
\end{proof}

\begin{lemma} \label{l1.13} Suppose that $(M,g)$ is closed and oriented
Riemannian manifold. Then for any vector field $X\in C^{\infty}(M,TM)$, we have
\begin{equation}
\int_{M}\left\langle\nabla_{X}X,X\right\rangle dV=-\frac{1}{2}
\int_{M}{\rm div}(X)|X|^{2}dV.\label{1.32}
\end{equation}
\end{lemma}

\begin{proof} Indeed, using $\left(\nabla_{X}X\right)^{j}=X^{i}\nabla_{i}X^{j}$ we have
\begin{eqnarray*}
\int_{M}\left\langle\nabla_{X}X,X\right\rangle dV&=&\int_{M}
\left(\nabla_{X}X\right)^{j}X_{j}\!\ dV \ \ = \ \ \int_{M}X^{i}\nabla_{i}X^{j}
\cdot X_{j}\!\ dV\\
&=&\int_{M}\nabla_{i}X^{j}(X^{i}X_{j})\!\ dV \ \ = \ \ -\int_{M}X^{j}
\nabla_{i}(X^{i}X_{j})\!\ dV\\
&=&-\int_{M}X^{j}\left[{\rm div}(X)X_{j}+X^{i}\nabla_{i}X_{j}\right]dV\\
&=&-\int_{M}{\rm div}(X)|X|^{2}dV-\int_{M}X^{i}X^{j}\nabla_{i}X_{j}\!\ dV\\
&=&-\int_{M}{\rm div}(X)|X|^{2}dV-\int_{M}\left\langle\nabla_{X}X,X\right\rangle dV.
\end{eqnarray*}
Arranging the terms yields (\ref{1.32}).
\end{proof}

The similar result was considered by Wilson \cite{Wilson} for the standard 
metric on $\mathbf{R}^{3}$.

\subsection{A connection to Kazdan-Warner-Bourguignon-Ezin identity}
\label{subsection1.7}

If $(M,g)$ is a closed Riemannian manifold with $m\geq2$ and if $X$ is a Killing vector field, then
\begin{equation}
\int_{M}\langle\nabla R, X\rangle\!\ dV=0.\label{1.40}
\end{equation}
This identity (actually holds for any conformal Killing vector fields) was
proved by Bourguignon and Ezin \cite{BE} and the surface case is the
classical Kazdan-Warner identity \cite{KW}. For convenience, we call such an 
identity as {\it {\rm KWBE} identity}. For its application to Ricci flow we refer readers to \cite{CLN}. In this subsection we study the asymptotic behavior of the KWBE identity under the flow (\ref{1.4}).

For any vector field $X$, we define the {\it {\rm KWBE} functional} as
\begin{equation*}
\mathcal{I}(X):=\int_{M}\langle\nabla R, X\rangle\!\ dV.
\end{equation*}
Then, where $X_{t}=X^{i}\frac{\partial}{\partial x^{i}}$,
\begin{eqnarray*}
\frac{d}{dt}\mathcal{I}(X_{t})&=&\int_{M}\nabla_{i}R\left(\Delta X^{i}
+\nabla^{i}{\rm div}(X_{t})+R^{i}_{j}X^{j}\right)dV\\
&=&\int_{M}\nabla_{i}R\cdot\Delta X^{i}\!\ dV-\int_{M}\Delta R\cdot{\rm div}
(X_{t})\!\ dV
+\int_{M}R_{ij}X^{i}\nabla^{j}R\!\ dV.
\end{eqnarray*}
Using the commutative formula $\nabla\Delta R=\Delta\nabla R-{\rm Rc}(\nabla R,\cdot)$ yields
\begin{eqnarray*}
\int_{M}\nabla_{i}R\cdot\Delta X^{i}dV&=&\int_{M}\langle X_{t},\Delta\nabla 
R\rangle\!\ dV\\
&=&\int_{M}\langle X_{t},\nabla\Delta R+{\rm Rc}(\nabla R,\cdot)\rangle\!\ dV\\
&=&-\int_{M}\Delta R\cdot{\rm div}(X_{t})\!\ dV
+\int_{M}R_{ij}X^{i}\nabla^{j}R\!\ dV
\end{eqnarray*}
and therefore
\begin{equation}
\frac{d}{dt}\mathcal{I}(X_{t})=-2\int_{M}\Delta R\cdot{\rm div}(X_{t})\!\ dV
+2\int_{M}{\rm Rc}(X_{t},\nabla R)\!\ dV.\label{1.34}
\end{equation}
The last term on the right-hand side of (\ref{1.34}) can be simplified by
\begin{eqnarray*}
\int_{M}\nabla^{i}R\left(X^{j}R_{ij}\right)dV&=&-\int_{M}R\left(\nabla^{i}X^{j}\cdot R_{ij}+X^{j}\cdot\frac{1}{2}\nabla_{j}R\right)dV\\
&=&-\int_{M}RR_{ij}\nabla^{i}X^{j}\!\ dV-\frac{1}{2}\int_{M}RX^{j}\nabla_{j}R\!\ dV.
\end{eqnarray*}
We also have
\begin{eqnarray*}
\int_{M}RX^{j}\nabla_{j}R\!\ dV&=&-\int_{M}\nabla_{j}(RX^{j})R\!\ dV\\
&=&-\int_{M}RX^{j}\nabla_{j}R\!\ dV-\int_{M}R^{2}{\rm div}(X_{t})\!\ dV
\end{eqnarray*}
so that
\begin{equation}
\int_{M}RX^{j}\nabla_{j}R\!\ dV=-\frac{1}{2}\int_{M}R^{2}
{\rm div}(X_{t})dV.\label{1.35}
\end{equation}

From (\ref{1.34}), (\ref{1.35}), (\ref{1.1}) and Theorem \ref{t1.8}, we arrive at

\begin{proposition}\label{p1.14} If $(M,g)$ is a closed Riemannian manifold and $X_{t}$ is a solution to (\ref{1.4}), then
\begin{equation}
\frac{d}{dt}\mathcal{I}(X_{t})=2\int_{M}\left(-\Delta+\frac{R}{4}\right)
R\cdot{\rm div}(X_{t})\!\ dV-2\int_{M}R\left\langle{\rm Rc},{\bf Def}(X_{t})
\right\rangle dV.\label{1.36}
\end{equation}
In particular,
\begin{equation}
\lim_{t\to\infty}\frac{d}{dt}\mathcal{I}(X_{t})=0.\label{1.37}
\end{equation}
\end{proposition}

This proposition gives the limiting behavior of $\frac{d}{dt}\mathcal{I}(X_{t})$. However, the pointwise behavior of $\frac{d}{dt}\mathcal{I}(X_{t})$ is very complicated.
For example, we can find a compact Riemannian manifold such that
$\frac{d}{dt}\mathcal{I}(X_{t})>0$ or $<0$ depending on the choice
of the initial vector fields.

\begin{corollary}\label{c1.15} Suppose that $(M,g)$ is a closed $m$-dimensional
Einstein manifold with $m\geq3$.

\begin{itemize}

\item[(a)] When $m=4$ or the scalar curvature of $g$ vanishes identically,
$\frac{d}{dt}\mathcal{I}(X_{t})=0$ for all $t$, where $X_{t}$ is the solution
of (\ref{1.4}) with any given initial vector field $X$.

\item[(b)] When $m\neq 4$ and the scalar curvature of $g$ does not
vanish identically, there exists a vector field $X$ such that $\frac{d}{dt}\mathcal{I}(X_{t})>0$ for all $t$, where $X_{t}$
is the solution of (\ref{1.4}) with the initial vector field $X$.

\item[(c)] When $m\neq 4$ and the scalar curvature of $g$ does not
vanish identically, there exists a vector field $X$ such that $\frac{d}{dt}\mathcal{I}(X_{t})<0$ for all $t$, where $X_{t}$
is the solution of (\ref{1.4}) with the initial vector field $X$.

\end{itemize}

\end{corollary}

\begin{proof} By assumption we have ${\rm Ric}=\frac{R}{m}g$ and $R$
is constant. Using (\ref{1.36}) we obtain
\begin{eqnarray}
\frac{d}{dt}\mathcal{I}(X_{t})&=&
\int_{M}\frac{R^{2}}{2}\cdot{\rm div}(X_{t})\!\ dV
-2\int_{M}\frac{R^{2}}{m}{\rm div}(X_{t})\!\ dV\nonumber\\
&=&\int_{M}\frac{m-4}{2m}R^{2}\cdot{\rm div}(X_{t})\!\ dV.\label{1.38}
\end{eqnarray}
The part (a) follows immediately.

For parts (b) and (c), we may assume that
$m>4$, otherwise we can consider $-X_{t}$. According to the evolution
equation (\ref{1.16}) yields
\begin{equation*}
\partial_{t}{\rm div}(X_{t})=
2\Delta{\rm div}(X_{t})
+\frac{2R}{m}{\rm div}(X_{t})
\end{equation*}
which can be written as
\begin{equation*}
\partial_{t}\left(e^{-\frac{2R}{m}t}{\rm div}(X_{t})\right)
=2\Delta\left(e^{-\frac{2R}{m}t}{\rm div}(X_{t})\right).
\end{equation*}
By the maximum principle,
\begin{equation}
e^{\frac{2R}{m}t}\min_{M}{\rm div}(X)
\leq{\rm div}(X_{t})\leq e^{\frac{2R}{m}t}\max_{M}{\rm div}(X), \ \ \
X:=X_{0}.\label{1.39}
\end{equation}
Given any fixed vector field $X'$, let $f$
be a smooth function on $M$ satisfying $\Delta f={\rm div}(X')-1$. Then
\begin{equation*}
{\rm div}(X'-\nabla f)={\rm div}(X')-\Delta f=1>0
\end{equation*}
on $M$. Choose $X:=X'-\nabla f$. Then
\begin{equation*}
{\rm div}(X_{t})\geq e^{\frac{2R}{m}t}>0 \ \ \ \text{on} \ M
\end{equation*}
for all $t$. Substituting this into (\ref{1.38}) we arrive at
\begin{equation*}
\frac{d}{dt}\mathcal{I}(X_{t})
\geq\frac{m-4}{2m}e^{\frac{2R}{m}t}\int_{M}R^{2}\!\ dV
\end{equation*}
where we used $m>4$. Since the scalar curvature $R$ does not vanishes
identically, the $L^{2}$-norm of $R$ must be positive and consequently,
$\frac{d}{dt}\mathcal{I}(X_{t})>0$ for all $t$.

Similarly, we can prove part (c).
\end{proof}

\section{A conjecture to the flow and its application}\label{section2}

Before stating a conjecture to the flow (\ref{1.4}), we shall look at a simple case
that $(M,g)$ is an Einstein manifold with positive sectional curvature and the
solution of (\ref{1.4}) is the sum of the initial vector field and a gradient
vector field. That is, we assume
\begin{equation*}
R_{ij}=\frac{R}{m}g_{ij}, \ \ \ m\geq3, \ \ \
X_{t}=X+\nabla f_{t},
\end{equation*}
where $f_{t}$ are some functions on $M$. By a theorem of Schur, the scalar curvature $R$ must be a constant. In this
case the flow (\ref{1.4}) is equivalent to
\begin{equation}
\nabla\left(\partial_{t}f_{t}-2\Delta f_{t}-\frac{2R}{m}f_{t}\right)
=X^{\dag},\label{2.1}
\end{equation}
where
\begin{equation}
X^{\dag}:=
\Delta X+\nabla({\rm div}(X))+{\rm Ric}^{\sharp}(X)\label{2.2}
\end{equation}
is the vector field associated to $X$. Clearly that the operator $\dag$ is not
self-adjoint on the space of vector fields, with respect to the $L^{2}$-inner product with respect to $(M,g)$.

\subsection{Einstein manifolds with positive scalar curvature}

If $(M,g)$ is an $m$-dimensional Einstein manifold with positive scalar curvature,
then we can prove that the limit vector field converges to a nonzero Killing
vector field, provided the initial vector field satisfying some conditions. We
first give a $L^{2}$-estimate for $f_{t}$.

\begin{proposition} \label{p2.1}Suppose that $(M,g)$ is an $m$-dimensional closed and orientable
Einstein manifold with positive scalar curvature $R$, where $m\geq3$. Let $X$ be a nonzero vector field satisfying $X^{\dag}=\nabla\varphi_{X}$ for
some smooth function $\varphi_{X}$ on $M$. Then for any given constant $c$, the equation
\begin{equation}
\partial_{t}f_{t}=2\Delta f_{t}
+\frac{2R}{m}f_{t}+\varphi_{X}, \ \ \ f_{0}=c,\label{2.3}
\end{equation}
exists for all time. Moreover,
\begin{itemize}

\item[(i)] we have
\begin{equation}
\int_{M}f_{t}\!\ dV=\left[c\cdot {\rm Vol}(M,g)+\frac{m}{2R}\int_{M}
\varphi_{X}\!\ dV\right]e^{\frac{2R}{m}t}
-\frac{m}{2R}\int_{M}\varphi_{X}\!\ dV.\label{2.4}
\end{equation}
Setting
\begin{equation*}
c_{X}:=-\frac{m}{2R\cdot{\rm Vol}(M,g)}\int_{M}\varphi_{X}\!\ dV,
\end{equation*}
yields
\begin{equation*}
\int_{M}f_{t}\!\ dV=-\frac{m}{2R}\int_{M}\varphi_{X}\!\ dV, \ \ \ \text{if} \ c=c_{X}.
\end{equation*}

\item[(ii)] if we choose the nonzero function $\varphi_{X}$ so that its integral over $M$ is zero and $f_{0}=0$, then $\int_{M}f_{t}\!\ dV=0$ and the $L^{2}$-norm of $f_{t}$ is bounded by
\begin{equation}
||f_{t}||_{2}\leq\frac{||\varphi_{X}||_{2}}{2
\left(\lambda_{1}-\frac{R}{m}\right)}
-\frac{||\varphi_{X}||_{2}}{2
\left(\lambda_{1}-\frac{R}{m}\right)}
e^{-2\left(\lambda_{1}
-\frac{R}{m}\right)t},\label{2.5}
\end{equation}
where $||\cdot||_{2}$ means $||\cdot||_{L^{2}(M,g)}$ the $L^{2}$-norm with
respect to $(M,g)$, and $\lambda_{1}$ stands for the first nonzero eigenvalue of $(M,g)$.

\end{itemize}
\end{proposition}

For a moment, we put
\begin{equation*}
a(t):=\int_{M}f_{t}\!\ dV, \ \ \ b(t):=\int_{M}|f_{t}|^{2}dV.
\end{equation*}
Then, the equation (\ref{2.3}) implies that
\begin{equation*}
a'(t)=\frac{2R}{m}a(t)+\int_{M}\varphi_{X}\!\ dV,
\end{equation*}
and
\begin{eqnarray*}
b'(t)&=&-4\int_{M}|\nabla f_{t}|^{2}dV+\frac{4R}{m}b(t)+2\int_{M}f_{t}
\varphi_{X}\!\ dV\\
&\leq&-4\left(\lambda_{1}-\frac{R}{m}\right)b(t)+2b^{1/2}(t)||\varphi_{X}
||_{2}.
\end{eqnarray*}
By a theorem
of Lichnerowicz, we have that $\lambda_{1}\geq\frac{R}{m-1}>\frac{R}{m}$. Hence
(\ref{2.3}) and (\ref{2.4}) follow immediately.

Consequently, we have the following

\begin{theorem} \label{t2.2}Suppose that $(M,g)$ is an $m$-dimensional closed
and orientable Einstein manifold with positive scalar curvature $R$, where $m\geq3$. Let $X$ be a nonzero vector field satisfying the following two conditions:
\begin{itemize}

\item[(i)] $X^{\dag}$ is a gradient vector field, and

\item[(ii)] $X$ is not a gradient vector field.

\end{itemize}
Then the flow (\ref{1.4}) with initial value $X$ converges uniformly
to a nonzero Killing vector field.
\end{theorem}

\subsection{A conjecture and its applications}

By Bochner's theorem, any Killing vector field on a closed and orientable Riemmanian
manifold with negative Ricci curvature is trivial. Hence, based on a result
in the Einstein case, we propose the following
conjecture.

\begin{conjecture} \label{c2.3}Suppose that $M$ is a closed Riemannian manifold with positive sectional curvature. For some initial vector field and a certain Riemannian metric $g$ of positive sectional curvature, the
flow (\ref{1.4}) converges uniformly to a nonzero Killing vector field with respect to $g$.
\end{conjecture}

Our study shows that we may need to change to a new metric, which
still has positive sectional curvature, to get the nonzero limit
which is a Killing vector field with respect to this new metric. For
this purpose we have computed variations of the functional $\mathfrak{L}$ or
$\mathcal{E}$ relative to the new metric, as well as the Perelman-type
functional for our flow.

Obviously a solution of this conjecture immediately answers the following
long-standing question of Yau \cite{SY}.

\begin{question} \label{q2.4}Does there exist an effective $\mathbb{S}^{1}$-action
on a closed manifold with positive sectional curvature?
\end{question}

Assuming Conjecture \ref{c2.3}, we can deduce several important corollaries. We
first recall the well-known Hopf's conjectures.

\begin{conjecture} \label{c2.5} If $M$ is a closed and even dimensional
Riemannian manifold with positive sectional curvature, then the Euler
characteristic number of $M$ is positive, i.e., $\chi(M)>0$.
\end{conjecture}

\begin{conjecture} \label{c2.6} On $\mathbb{S}^{2}\times\mathbb{S}^{2}$
there is no Riemannian metric with positive sectional curvature.
\end{conjecture}

For the recent development of Hopf's conjectures, we refer to \cite{P, SY}. A simple argument shows that Conjecture \ref{c2.5} and \ref{c2.6} follow
from Conjecture \ref{c2.3}.

\begin{corollary} \label{c.2.7}
Conjecture \ref{c2.3} implies Conjecture \ref{c2.5}.
\end{corollary}

\begin{proof} From \cite{K} we know that the Killing vector field $X$ must have
zero, and the zero sets consist of finite number of totally geodesic
submanifolds $\{M_{i}\}$ of $M$ with the induced Riemannian metrics. Moreover each $M_{i}$ is even dimensional and has positive sectional curvature. Hence
we have $\chi(M)=\sum_{i}\chi(M_{i})$. By induction, we obtain $\chi(M)>0$.
\end{proof}

Hsiang and Kleiner \cite{HK} showed that if $M$ is a $4$-dimensional closed
Riemannian manifold with positive sectional curvature, admitting
a nonzero Killing vector field, then $M$ is homeomorphic to $\mathbb{S}^{4}$
or $\mathbb{CP}^{2}$. Consequently, $\mathbb{S}^{2}\times\mathbb{S}^{2}$ does
not admit a Riemannian metric, whose sectional curvature is positive, with a
nontrivial Killing vector field. Therefore

\begin{corollary} \label{c2.8} Conjecture \ref{c2.3} implies Conjecture \ref{c2.6}.
\end{corollary}

\section{Variants geometric flows}\label{section3}

In the last section, we discuss several new geoemtric flows whose fixed points 
give Killing vector fields. Recall the notions in \cite{Lott03}. Let 
$(M,g)$ be a closed and orientable Riemannian manifold of dimension $m$ and $\phi$ a positive smooth function on $M$. Define
\begin{equation}
\widetilde{{\rm Ric}}_{\infty}:={\rm Ric}-{\rm Hess}(\ln\phi)\label{3.1}
\end{equation}
the Bakry-\'Emery Ricci tensor field. For any smooth tensor field $T$ 
on $M$ consider the weighted $L^{2}$-inner product given by
\begin{equation}
\langle T,T\rangle_{\phi}:=\int_{M}(T,T)\phi\!\ dV\label{3.2}
\end{equation}
and let us denote $\tilde{\delta}$ the formal adjoint of $d$ with
respect to this inner product. Then
\begin{equation}
\tilde{\delta}=\delta-i_{(d\ln\phi)^{\#}}\label{3.3}
\end{equation}
where $\delta$ is the usual formal adjoint of $d$ and $(d\ln\phi)^{\#}$
stands for the corresponding vector field of the $1$-form $d\ln\phi$.

In \cite{Lott03}, Lott obtained the following Bochner formula (where $\omega$
is a $1$-form):
\begin{equation}
\langle d\omega,d\omega\rangle_{\phi}+\langle\tilde{\delta}\omega,
\tilde{\delta}\omega\rangle_{\phi}-\langle\nabla\omega,\nabla\omega\rangle_{\phi}
=\langle\widetilde{{\rm Ric}}_{\infty}\omega,\omega\rangle_{\phi}\label{3.4}
\end{equation}
or
\begin{equation}
\langle\nabla\omega,\nabla\omega\rangle_{\phi}
+\langle\tilde{\delta}\omega,\tilde{\delta}\omega\rangle_{\phi}
-\langle\omega,\widetilde{{\rm Ric}}_{\infty}\omega\rangle_{\phi}
=\langle\mathcal{L}_{\omega^{\#}}g,\mathcal{L}_{\omega^{\#}}g\rangle_{\phi}\label{3.5}
\end{equation}
where $\mathcal{L}$ means the Lie derivative. Let $X:=\omega^{\#}$ or 
$X_{\flat}=\omega$ in (\ref{3.5}) we obtain
\begin{equation}
\int_{M}|\mathcal{L}_{X}g|^{2}\phi\!\ dV
=\int_{M}\left[|\nabla X|^{2}+|\tilde{\delta}X_{\flat}|^{2}
-\widetilde{{\rm Ric}}_{\infty}(X,X)\right]\phi\!\ dV.\label{3.6}
\end{equation}

\subsection{New criterion: I}\label{subsection3.1}

Given a smooth function $f$ on $M$, set
\begin{equation}
\phi:=e^{f}, \ \ \ \ln\phi=f\label{3.7}
\end{equation}
and define
\begin{eqnarray*}
{\rm Ric}_{f}&:=&\widetilde{{\rm Ric}}_{\infty} \ \ = \ \
{\rm Ric}-{\rm Hess}(f),\\
{\rm div}_{f}&:=&-\tilde{\delta} \ \ = \ \ -\delta+i_{\nabla f} \ \ = \ \
{\rm div}+i_{\nabla f}.
\end{eqnarray*}
For any smooth vector field $X$ we have
\begin{equation*}
e^{-f}{\rm div}\left(e^{f}X\right)=e^{-f}\left(e^{f}{\rm div}(X)
+e^{f}\langle\nabla f,X\rangle\right)={\rm div}(X)+\langle\nabla f,X\rangle
\end{equation*}
which implies that
\begin{equation}
{\rm div}_{f}=\frac{1}{e^{f}}{\rm div}\left(e^{f}\!\ \right),
\label{3.8}
\end{equation}
a weighted divergence in the sense of \cite{G}. Therefore the identity (\ref{3.6}) can be rewritten as
\begin{equation}
\int_{M}|\mathcal{L}_{X}g|^{2}e^{f}dV
=\int_{M}\left[|\nabla X|^{2}+|{\rm div}_{f}(X)|^{2}
-{\rm Ric}_{f}(X,X)\right]e^{f}dV.\label{3.9}
\end{equation}
On the other hand, we have
\begin{eqnarray*}
\int_{M}|\nabla X|^{2}e^{f}dV&=&\int_{M}\nabla_{i}X_{j}
\left(e^{f}\nabla^{i}X^{j}\right)\!\ dV\\
&=&-\int_{M}X_{j}\left(\nabla_{i}f\nabla^{i}X^{j}
+\Delta X^{j}\right)e^{f}dV\\
&=&-\int_{M}\langle X,\Delta_{f}X\rangle e^{f}dV
\end{eqnarray*}
where
\begin{equation*}
\Delta_{f}X^{j}:=\Delta X^{j}+\nabla_{i}f\nabla^{i}X^{j}.
\end{equation*}
Similarly,
\begin{eqnarray*}
\int_{M}|{\rm div}_{f}(X)|^{2}e^{f}dV&=&
\int_{M}{\rm div}_{f}(X)\left(e^{f}{\rm div}_{f}(X)\right)dV\\
&=&\int_{M}e^{-f}{\rm div}\left(e^{f}X\right)\left(e^{f}
{\rm div}_{f}(X)\right)dV\\
&=&-\int_{M}\langle X,\nabla{\rm div}_{f}(X)\rangle e^{f}dV.
\end{eqnarray*}
Hence the identity (\ref{3.9}) implies
\begin{equation}
\int_{M}|L_{X}g|^{2}e^{f}dV
=-\int_{M}\left\langle X,\Delta_{f}X+\nabla{\rm div}_{f}(X)+{\rm Ric}_{f}(X)\right\rangle e^{f}dV.\label{3.10}
\end{equation}
The above identity shows that the Euler-Lagrange equation for the functional
\begin{equation*}
X\mapsto \int_{M}|L_{X}g|^{2}e^{f}dV
\end{equation*}
is
\begin{equation}
\Delta_{f}X+\nabla{\rm div}_{f}(X)+{\rm Ric}_{f}(X)=0.\label{3.11}
\end{equation}
We now simplify the equation (\ref{3.11}). Compute
\begin{eqnarray*}
\Delta_{f}X^{i}&=&\Delta X^{i}+\nabla_{j}f\nabla^{j}X^{i},\\
\nabla^{i}{\rm div}_{f}(X)&=&\nabla^{i}\left(e^{-f}{\rm div}
\left(e^{f}X\right)\right)\\
&=&\nabla^{i}\left({\rm div}(X)+\langle\nabla f,X\rangle\right)\\
&=&\nabla^{i}{\rm div}(X)+\nabla^{i}(X^{j}\nabla_{j}f)\\
&=&\nabla^{i}{\rm div}(X)+\nabla^{i}X^{j}\nabla_{j}f
+X^{j}\nabla^{i}\nabla_{j}f.
\end{eqnarray*}
Consequently,
\begin{equation}
\Delta_{f}X^{i}+\nabla^{i}{\rm div}_{f}(X)
=\Delta X^{i}+\nabla^{i}{\rm div}(X)+\nabla_{j}f(L_{X}g)^{ij}+X_{j}
\nabla^{i}\nabla^{j}f.\label{3.12}
\end{equation}
Plugging (\ref{3.12}) into (\ref{3.11}) and noting the definition of ${\rm Ric}_{f}$ yields
\begin{equation}
0=\Delta X^{i}+\nabla^{i}{\rm div}(X)+R^{i}{}_{j}X^{j}+\nabla_{j}f(L_{X}g)^{ij}.
\label{3.13}
\end{equation}

As in \cite{Y1}, we can prove the following 

\begin{theorem}\label{t3.1} Given any smooth function $f$ on a closed
orientable Riemanian manifold $(\mathcal{M},g)$. A smooth vector field $X$ is Killing if and only if it satisfies (\ref{3.13}). When $f\equiv0$, it reduces to
the classical criterion of Yano.
\end{theorem}

\begin{proof} Suppose $X$ is Killing. Then $L_{X}g=0$ and $\Delta X
+\nabla{\rm div}(X)+{\rm Ric}(X)=0$ by Yano's theorem. These two equations
immediately imply (\ref{3.13}). Conversely, if $X$ is a smooth vector field
satisfying (\ref{3.13}), then it also satisfies (\ref{3.11}) and then
\begin{equation*}
\int_{M}|\mathcal{L}_{X}g|^{2}e^{f}dV=0
\end{equation*}
according to (\ref{3.10}). Since $\mathcal{L}_{X}g$ is symmetric, by choosing a suitable
coordinates we can diagonalize $\mathcal{L}_{X}g$ into ${\rm diag}(\lambda_{1},\cdots,
\lambda_{m})$ so that each $\lambda_{i}$ must be identically zero. Hence $\mathcal{L}_{X}
g\equiv0$ and $X$ is Killing.
\end{proof}

The above theorem suggests us to consider the following flow
\begin{equation}
\partial_{t}X^{i}=\Delta X^{i}+\nabla^{i}{\rm div}(X)
+R^{i}{}_{j}X^{j}+\nabla_{j}f(\mathcal{L}_{X}g)^{ij}\label{3.14}
\end{equation}
for a given smooth function $f\in C^{\infty}(M)$, or consider
a nonlinear flow
\begin{equation}
\partial_{t}X^{i}=\Delta X^{i}+\nabla^{i}{\rm div}(X)
+R^{i}{}_{j}X^{j}+\nabla_{j}{\rm div}(X)(\mathcal{L}_{X}g)^{ij}.\label{3.15}
\end{equation}

As in the proof of Theorem \ref{t1.8}, we have

\begin{theorem}\label{t3.2} Let $(M,g)$ be a closed orientable Riemannian manifold, $f$ a smooth function on $M$, and $X$ a smooth vector field on $M$. Then the flow (\ref{3.14}) starting with the initial data $X$
smoothly converges to a Killing vector field $X_{\infty}$.
\end{theorem}

\begin{proof} By replacing ${\rm div}, \Delta, dV, {\rm Ric}$ by ${\rm div}_{f},
\Delta_{f}, e^{f}dV, {\rm Ric}_{f}$ in the argument of Theorem \ref{t1.8}, we can show that $\int_{M}|X_{t}|^{2}e^{f}dV$ is decreasing, $\int_{M}
 |\partial_{t}X_{t}|^{2}e^{f}dV\to0$, and then by the
same method $X_{t}$ smoothly converges to a smooth vector field
$X_{\infty}$ satisfying (\ref{3.13}). By Theorem \ref{t3.1}, $X_{\infty}$
must be Killing.
\end{proof}

\subsection{New criterion: II}\label{subsection3.2}

The second new criterion is based on the following identity
\begin{equation}
\int_{M}\left[(\mathcal{L}_{X}g)(X,X)+\frac{1}{2}
{\rm div}(X)|X|^{2}\right]dV=0\label{3.16}
\end{equation}
for any smooth vector field $X$ on $M$. Since $2(\mathcal{L}_{X}g)_{ij}=\nabla_{i}X_{j}
+\nabla_{j}X_{i}$, to prove (\ref{3.16}), we suffice to show that
\begin{equation*}
\int_{M}\left[X^{i}X^{j}\nabla_{i}X_{j}+\frac{1}{2}{\rm div}(X)|X|^{2}
\right]dV=0.
\end{equation*}
Actually,
\begin{eqnarray*}
\int_{M}X^{i}X^{j}\nabla_{i}X_{j}\!\ dV&=&-\int_{M}X_{j}\nabla_{i}(X^{i}X^{j})\!\ dV\\
&=&-\int_{M}X_{j}\left[{\rm div}(X)X^{j}
+X^{i}\nabla_{i}X^{j}\right]dV\\
&=&-\int_{M}{\rm div}(X)|X|^{2}dV-\int_{M}X^{i}X_{j}\nabla_{i}X^{j}\!\ dV
\end{eqnarray*}
which yields
\begin{equation*}
\int_{M}X^{i}X^{j}\nabla_{i}X_{j}\!\ dV=-\frac{1}{2}\int_{M}{\rm div}(X)|X|^{2}\!\ dV.
\end{equation*}

The second new criterion can be stated as follows.

\begin{theorem}\label{t3.3} A smooth vector field $X$ on a closed orientable Riemannian 
manifold $(M,g)$ is a Killing vector field if and only if it satisfies 
\begin{equation}
0=\Delta X+\nabla{\rm div}(X)+{\rm Ric}(X,\cdot)+(\mathcal{L}_{X}g)(X,\cdot)
+\frac{1}{2}{\rm div}(X) X.\label{3.17}
\end{equation}
\end{theorem}

\begin{proof} If $X$ is Killing, then ${\rm div}(X)=\mathcal{L}_{X}g=0$ and 
hence (\ref{3.17}) reduces to Yano's classical result. Conversely, suppose a 
smooth vector field $X$ satisfies (\ref{3.17}). Multiplying (\ref{3.17}) by $X$ 
and integrating over $M$, we obtain
\begin{eqnarray}
0&=&-\int_{M}[|\nabla X|^{2}+|{\rm div}(X)|^{2}-{\rm Ric}(X,X)]dV\nonumber\\
&&+ \ \int_{M}\left[(\mathcal{L}_{X}g)_{ij}X^{i}X^{j}+\frac{1}{2}
{\rm div}(X)|X|^{2}\right]dV.\label{3.18}
\end{eqnarray}
The second integral on the right-hand side equals zero by 
the identity (\ref{3.16}), and consequently, (\ref{3.18}) is equivalent to 
$\mathcal{E}(X)=0$, where $\mathcal{E}(X)$ was defined in (\ref{1.8}). By a 
result of Watanabe \cite{W}, $X$ must be Killing.
\end{proof}

Theorem \ref{t3.3} also suggests a nonlinear equation
\begin{equation}
\partial_{t}X=\Delta X+\nabla{\rm div}(X)+{\rm Ric}(X,\cdot)
+(\mathcal{L}_{X}g)(X,\cdot)+\frac{1}{2}{\rm div}(X)X.\label{3.19}
\end{equation}

We note that the flows (\ref{1.4}) and (\ref{3.14}) linear, while the flows (\ref{3.15}) and (\ref{3.19}) are nonlinear. We will later study those flows and applications to geometry.
\\
\\


{\bf Acknowledgements.}The authors would like to thank Professors
Shing-Tung Yau, Cliff Taubes, Youde Wang, Xinan Ma and
Hongwei Xu, with whom we have discussed. We also thanks doctors Yan
He and Jianming Wan for several discussion.




\end{document}